\documentclass[12pt]{article} 
\usepackage{amsmath}
\DeclareMathOperator{\sgn}{sgn}

\usepackage{amssymb}
\usepackage{amsthm}  
\usepackage{caption} 
\usepackage{bigints}
\usepackage{textcomp}    				
\usepackage[T1]{fontenc} 				
\usepackage{marvosym}    				
\usepackage[sc]{mathpazo}  	 			
 \usepackage[english]{babel} 
 \usepackage{float} 
 \usepackage{graphics} 
 \usepackage{graphicx} 
\usepackage[utf8]{inputenc}
\usepackage{mathtools}
\usepackage{authblk} 					
\usepackage[usenames]{xcolor}  			
\usepackage{lastpage} 					

\usepackage{enumerate}	 				

\usepackage{hyperref} 					

\usepackage{capt-of} 
\usepackage{calrsfs}
\usepackage{lipsum}  
\usepackage{epsfig}
\newcommand{\R}{\mathbf{R}}
\newcommand{\C}{\mathbf{C}}

\newtheorem{theorem}{Theorem}

\newtheorem{corollary} {Corollary}

\newtheorem{lemma}{Lemma}

\theoremstyle{definition}

\theoremstyle{definition}
\newtheorem{remark}{Remark}

\theoremstyle{definition}
\newtheorem{example}{Example}

\numberwithin{equation}{section}
\numberwithin{table}{section}
\numberwithin{figure}{section}

\DeclareMathAlphabet{\pazocal}{OMS}{zplm}{m}{n}

\numberwithin{equation}{section}

\title{On the non-existence of oscillation numbers in Sturm-Liouville theory }
\author[1]{\textbf{Angelo B. Mingarelli}\footnote{Corresponding author. Email: angelo@math.carleton.ca}}

\affil[1]{School of Mathematics and Statistics, Carleton University, Ottawa, Canada}

\newcommand{\doi}[1]{\url{https://doi.org/#1}}

\DeclareMathOperator\supp{supp}

\makeatletter
\renewcommand{\maketitle}{\bgroup\setlength{\parindent}{0pt}

\vspace{1truecm}
\begin{center}{\vbox{\titlefont\@title}}\end{center}
\vspace{0.5truecm}
\begin{center}{\@author} \end{center}

\egroup
}

\renewcommand{\@fnsymbol}[1]{%
    \ifcase#1 \or {\,\Letter\!} \or\textasteriskcentered\or \textasteriskcentered\textasteriskcentered 
    \else\@ctrerr\fi}
\makeatother

\newcommand*{\titlefont}{\fontsize{18}{21.6}\selectfont\textbf}

\def\R{\mathbb{R}}

\begin{document}

\maketitle

\pagestyle{plain}

\begin{center}
\noindent
\begin{minipage}{0.85\textwidth}\parindent=15.5pt



{\small{
\noindent {\bf Abstract.}
We prove an old conjecture that relates the existence of non-real eigenvalues of Sturm-Liouville Dirichlet problems on a finite interval to the non-existence of oscillation numbers of its real eigenfunctions, [\cite{atm2}, p.104, Problems 3 and 5]. This extends to the general case, a previous result in \cite{alm1}, \cite{alm2} where it was shown that the presence of even one pair of non-real eigenvalues implies the non-existence of a positive eigenfunction (or ground state). We also provide estimates on the Haupt and Richardson indices and Haupt and Richardson numbers thereby complementing the original Sturm oscillation theorem with the Haupt-Richardson oscillation theorem discovered over 100 years ago with estimates on the missing oscillation numbers of the real eigenfunctions observed.
\smallskip

\small

}}

\end{minipage}
\end{center}

\section{Introduction}
The classical Sturm-Liouville boundary problem is one of the oldest and most important problems in applied mathematics. Its applications span the earliest ones from vibrating strings and heat conduction to quantum mechanics and beyond. The so-called problem with {\it fixed ends} (i.e., the homogeneous Dirichlet problem)  is expressed most commonly by asking for those values of $\lambda \in \C$ (called eigenvalues) such that the equation 
\begin{equation}\label{e1}
 Ly:= - (p(x)y^\prime)^\prime + q(x)\, y = \lambda w(x) y,\quad \quad y(a) = 0 = y(b),
\end{equation}
has a nontrivial solution (called an eigenfunction) satisfying the boundary conditions in \eqref{e1}. Motivated by its physical applications, in the period from 1829-1840 both Sturm and Liouville, separately and jointly, considered this question in the special cases where the coefficients $p, q, r$ are positive and continuous functions on $[a, b]$ and proved the existence of an infinite number of real eigenvalues that is bounded below, having no finite point of accumulation (no complex eigenvalues), an oscillation theorem for the eigenfunctions, and eventually an expansion theorem.  

The spectrum of \eqref{e1} comprises the collection of its eigenvalues. This spectrum can vary dramatically as the sign conditions on the coefficients are relaxed to the point where the whole complex plane can be the spectrum, \cite{atm}. For a brief history as to the development of the various cases we refer the reader to \cite{abm}. 

Classically, the Sturm oscillation theorem states that if the eigenvalues \{$\lambda_m$\} of \eqref{e1} are ordered in an increasing fashion so that $-\infty < \lambda_0 < \lambda_1< \cdots < \lambda_m < \cdots$, and $\lambda_m \to \infty$ as $m \to \infty$, then for every natural number $n$,  an eigenfunction corresponding to an eigenvalue $\lambda_n$ will have exactly $n$ zeros in $(a, b)$. This result had various extensions to the so-called {\it polar} and {\it orthogonal} cases (in older terminology) or {\it left-definite} and {\it right-definite}, respectively,  as they are known today. The case where $w(x)$ changes sign is commonly known as an {\it indefinite} case but, unfortunately, this can also include the polar case. Thus, in keeping with older terminology, we use the term {\bf non-definite} for the general case where both $q, w$ change their sign on $[a, b]$ and $p(x)>0$ there. Specifically, \eqref{e1} is in the non-definite case if each of the quadratic forms
$$\int_a^b (p|y^\prime|^2+q|y|^2)\, dx, \quad \quad \int_a^b w|y|^2\, dx,$$
are sign indefinite on the space of all functions $y$ such that $y, py^\prime$, ($p(x)>0$ a.e.) are absolutely continuous and satisfy the boundary conditions in \eqref{e1}.

The qualitative and spectral theory theory behind the non-definite case is still not completely understood. What is known of the spectrum in the non-definite case of \eqref{e1} is that it consists of a doubly infinite sequence of real eigenvalues accumulating only at $\pm \infty$ and there is an at most finite, though possibly empty, set of non-real eigenvalues. 

A major advance in the general non-definite case was made (apparently independently) by both Otto Haupt and R.G.D. Richardson in \cite{oh}, \cite{rgdr} (see \cite{abm} for more details). They proved an analog of Sturm's oscillation theorem in the non-definite case which we state here for the sake of completeness and that we call the {\bf Haupt-Richardson oscillation theorem}:

\begin{theorem}\label{hr}{\rm [Haupt-Richardson, \cite{abm}]} In the non-definite case of \eqref{e1} there is an integer $n_R \geq 0$ such that for each $n \geq n_R$ there are at least two real solutions of  \eqref{e1} having n zeros in (a, b) while for $n < n_R$ there are no real solutions having n zeros in (a, b). Furthermore, there is a possibly different integer $n_H \geq n_R$ such that for each $n \geq n_H$ there are only two solutions having exactly n zeros in $(a, b)$.
\end{theorem}

The usual Sturm oscillation theorem is subsumed in the preceding as the case where $n_R=n_H=0$. The quantities $n_R$, $n_H$ whose study form, in part, the body of this work will be defined next.

Let $\lambda > 0$. Referring to Theorem~\ref{hr}, the smallest natural number $n_R$ such that for every $n\geq n_R$ there is {\it at least} one solution of \eqref{e1} oscillating $n$ times in $(a, b)$ is called the {\bf  Richardson index}. On the other hand, the smallest natural number $n_H$ such that for every $n\geq n_H$ \eqref{e1} has {\it exactly} one solution oscillating $n$ times in $(a, b)$ is called the {\bf Haupt index}. To the best of our knowledge the only treatment of the Richardson index for non-definite problems appeared in \cite{alm1}, \cite{alm2}, \cite{aj}, \cite{km}, \cite{mkm}, \cite{km2} and the references therein. Similar definitions apply in the case where $\lambda <0$. 

Incidentally, the notion of a {\bf Richardson number} (it should have been called the {\it Haupt number} for consistency), is defined as that smallest eigenvalue  denoted by $\Lambda_R$, such that for all other eigenvalues $\lambda \geq \Lambda_R$ there is exactly {\bf exactly one} eigenfunction $y(x, \lambda)$ oscillating $n$ times for all sufficiently large $n$. This was originally mentioned in \cite{aj} and some results as to its estimates there were obtained in a special case there. The first estimation of the Richardson number for a problem with one turning point appeared in \cite{aj}. Estimations in the case of two turning points can be found in \cite{km} and the references therein.

In order to be consistent with existing nomenclature, the {\bf Haupt number} (it should really be called the Richardson number) must now be defined as that smallest eigenvalue of \eqref{e1} say, $\Lambda_H$, such that for all eigenvalues $\lambda \geq \Lambda_H$ we have {\bf at least one} eigenfunction $y(x, \lambda)$ oscillating $n$ times. That there can be more than one eigenfunction oscillating a fixed number of times is now well-known (see \cite{aj}, \cite{km}).

In this paper we will give general estimates for both the Haupt and Richardson indices thereby improving on those obtained in \cite{alm2}. We also give estimates for both the Haupt and Richardson numbers. Our results will be stated for positive eigenvalues only as analogous results for the negative eigenvalues may be stated readily by replacing $\lambda$ by $-\lambda$ and $r(x)$ by $-r(x)$ in the sequel. We leave these changes to the reader. 

Thus, after some basic notions and lemmas in Section~\ref{prel} which culminate with a rewriting of the minimum principle for real eigenvalues, we prove our main results in Section~\ref{main}. The main purpose of  Section~\ref{main} is to prove a conjecture, long left unanswered, [\cite{atm2}, p.104, Problem 5], which generalises a key result in \cite{alm2} where it was shown that the existence of just one pair of non-real eigenvalues implies the non-existence of a real eigenfunction that has one sign in $(a, b)$, i.e., $n_R \geq 1$. In Theorem~\ref{th4} we show that, essentially, $n_R \geq m+n,$ where $m$ is the number of pairs of non-conjugate distinct non-real eigenvalues of \eqref{e1} (also called ghost states) and $n$ is the number of real eigenvalues corresponding to degenerate ghost states (see the next section for definitions). That this value of $n < \infty$ is a consequence of results in \cite{abm2}. After an excursion into Lyapunov's inequality we present some bounds on all four numbers $n_H, n_R, \Lambda_H, \Lambda_R$ that complement the bounds in \cite{alm2}. 

\section{Preliminaries}\label{prel}
Much of the notation that follows regarding {\it ghosts} is due to W. Pauli \cite{kvm} in letters to K\"all\'en and Lee (1954) and, later on, W. Heisenberg (1955), see \cite{abm} for other references. An eigenfunction of \eqref{e1} corresponding to a non-real eigenvalue $\lambda$ will be called a {\bf complex ghost}, see \cite{abm}. Complex ghosts always satisfy
$$\int_a^b |y|^2\, w\, dx =0.$$
A complex ghost is said to be {\bf degenerate} if $$\int_a^b y^2\, w\, dx =0,$$ and {\bf non-degenerate} if the preceding integral is non-zero.

A real eigenfunction $u$ with corresponding real eigenvalue $\lambda$ will be called a {\bf degenerate real ghost} if it 
satisfies $$\int_a^b u^2\, w\, dx =0. $$  It will be called a {\bf non-degenerate real ghost} if 
$$ \lambda \int_a^b u^2\, w\, dx < 0.$$ The number of real ghosts (whether degenerate or not) is always finite, \cite{abm2}. The possible presence of such ghosts should not be surprising as we  allow $w(x)$ to change its sign on  $(a, b)$. Since the coefficients in \eqref{e1} are real, it is easy to see that complex ghosts occur in complex conjugate pairs, i.e., if $\lambda, \varphi$ is such an eigenvalue/eigenfunction pair, then $\bar{\lambda}, \bar{\varphi}$ is another such pair. By a {\bf ground state} of \eqref{e1} is meant a real eigenfunction $u$ with $u(x) \neq 0$ for $x \in (a, b)$. In the sequel we will generally use $u, v, \ldots$ for real eigenfunctions and $\varphi, \psi,\ldots$ for non-real eigenfunctions.

For the sake of simplicity we will sometimes assume that the functions $p, q, w \in C[a, b]$ (piecewise continuous is sufficient as is $L^1[a, b]$)  where $[a, b]=I \subset \R$ is a closed bounded interval, $p(x)>0$ for all $x\in I$ and non-empty support, $\supp w(x)$, of $w$. Thus, in addition to the usual conditions that $w(x)$ must change its sign on $I$, we also allow $w(x)\equiv 0$ on intervals (or sets of positive Lebesgue measure) in some cases. If $1/p, q, w \in L^1[a, b]$ and $|p(x)|<\infty$ a.e. in $[a, b]$ then it is well known (e.g., [\cite{atk0}, Chapter 8]) that solutions $y$ of \eqref{e1} and the quantity, $py^\prime$ are both in $AC[a, b]$ and \eqref{e1} is satisfied a.e. So, if $p(x)=1$, $q, r \in L^1[a, b]$ then all solutions are necessarily $C^1(a, b)$.

The first few lemmas are classical, easily shown, and so the proofs are omitted, \cite{abm}. In the sequel, unless otherwise noted, we will suppress the independent variable in the various differential or integral expressions for ease of use.
\begin{lemma}\label{lem1}
Let $\mu$ be a real eigenvalue of \eqref{e1}, $u$ a corresponding real eigenfunction, and let $\lambda$, $\varphi$, be a non-real eigenvalue/eigenfunction of \eqref{e1}. Then,
\begin{equation}\label{e2}
\int_a^b u\,\varphi\, w\, dx = 0\quad  \text{and}\quad  \int_a^b u\,\bar{\varphi}\, w\, dx = 0.
\end{equation}
\end{lemma}

\begin{lemma}\label{lem2}
Let $\lambda_1, \lambda_2$, be non-real eigenvalues of \eqref{e1} with $\lambda_1 \neq  \bar{\lambda}_2$ and $\varphi_1, \varphi_2$ their corresponding eigenfunctions. Then,
\begin{equation}\label{e3}
\int_a^b (p {\varphi_1}^\prime {\bar {\varphi_2}}^\prime+ q \varphi_1 \bar{\varphi_2})\, dx =0\quad \text{and} \quad 
\int_a^b \varphi_1\, \bar{\varphi}_2 \,w\, dx = 0.
\end{equation}
Similarly, whenever $\lambda_1 \neq  \lambda_2$, 
\begin{equation}\label{e4}
\int_a^b \varphi_1\, \varphi_2 \,w\, dx = 0.
\end{equation}
Finally, for any non-real eigenvalue/eigenfunction pair $\lambda, \varphi$, we always have both 
\begin{equation}\label{e5}
\int_a^b (p |{\varphi}^\prime|^2 + q |\varphi|^2)\, dx =0  \quad \text{and} \quad \int_a^b |\varphi|^2 \,w\, dx = 0.
\end{equation}
\end{lemma}
\begin{remark}
In other words, \eqref{e5} states that all non-real eigenfunctions corresponding to non-real eigenvalues are (complex) ghost states (or ghosts, for short). In addition, observe that both equalities in \eqref{e5} hold whenever $\varphi$ is a degenerate real ghost. 
\end{remark}

In the sequel we will always assume that $\mathcal{S} = \supp w$ has positive Lebesgue measure. 

\begin{lemma}\label{lem3}
For any $i, j$ satisfying $1 \leq i < j \leq n$, let $\lambda_i, \lambda_j$ with $\lambda_i \neq \lambda_j$ be any set of eigenvalues of \eqref{e1} with eigenfunctions $\varphi_i, \varphi_j$. Then the set $\{\varphi_1, \varphi_2, \ldots, \varphi_n\}$ of all such eigenfunctions is a linearly independent set over $\mathcal{S}$.
\end{lemma}

\begin{proof} 
Letting $\sum_{i=1}^n c_i\, \varphi_i =0,$ as usual and applying the operator $L$ we obtain, for $x\in \mathcal{S}$, 
\begin{equation}\label{e6}
\sum_{i=1}^n c_i\, \lambda_i w \varphi_i =0 \Longrightarrow \sum_{i=1}^n c_i\, \lambda_i  \varphi_i =0, \quad x\in \mathcal{S}.
\end{equation}
Applying $L$ once again to the right side of \eqref{e6} and iterating this procedure we obtain the system of $n$ equations in the $n$ unknowns, $c_i\varphi_i$, valid for $x \in \mathcal{S}$, and for $m= 0,1,2,..., n-1,$ i.e., 
\begin{equation}\label{e7}
\sum_{i=1}^n c_i\, \lambda_i^m  \varphi_i =0, \quad x\in \mathcal{S}.
\end{equation}
The determinant of the latter matrix system is the Vandermonde determinant, 
$$\prod_{1 \leq i< j \leq n} (\lambda_j - \lambda_i) \neq 0,$$
by hypothesis. It follows that all $c_i\varphi_i =0$, $x \in \mathcal{S}$, and so $c_i=0$ for all $i$.
\end{proof}
\begin{remark}
Clearly, in the case where $w$ is continuous, the result holds if $\supp w$ is replaced by any interval over which $w(x) \neq 0$. Furthermore, this lemma also holds if all the eigenvalues are real and distinct.
\end{remark}

The next result is an alternate characterization of the real eigenvalues of Sturm-Liouville problems as a minimum of a quadratic functional over an appropriate space, but without sign restrictions on the weighted square-integrable generally indefinite inner products. This allows for ghosts as defined earlier.

Fix a differentiable function $u$ and consider the ideal $\mathcal{V}$ generated by $u$ in the ring of all complex-valued differentiable functions, $\eta$, on $(a, b)$. Thus, 
$$\mathcal{V} = \{ \varphi : \varphi = u\eta, \eta \in C^1(a, b)\},$$
\begin{lemma}\label{lem4}
Let $\lambda_u \in \R$ be an eigenvalue of \eqref{e1} with real eigenfunction $u$.  Then, for any $\varphi \in \mathcal{V}$,
\begin{equation}\label{e8}
\int_a^b p|\varphi^\prime|^2 +q|\varphi|^2 \geq \lambda_u \int_a^b |\varphi|^2\, w,
\end{equation}
with equality holding iff $\varphi = \alpha u$, $\alpha \in \C$, on $I$.
\end{lemma}

\begin{proof}
First we note that eigenfunctions of \eqref{e1} can only have finitely many zeros on $I$ (otherwise the point of accumulation would imply, along with existence and uniqueness, that the eigenfunction would vanish identically in some interval about it).
A simple, though lengthy calculation shows that
 \begin{gather}\label{e9}
p|\varphi^\prime|^2 +q|\varphi|^2  = p (u^\prime)^2 |\eta|^2 + p u^2|\eta^\prime|^2 + qu^2|\eta|^2 + uu^\prime(\bar{\eta}p\eta^\prime + \eta p\bar{\eta}^\prime) 
\end{gather}
Simple integrations and use of \eqref{e1} show that
\begin{gather}
\int_a^b  uu^\prime(\bar{\eta}p\eta^\prime + \eta p\bar{\eta}^\prime) = \int_a^b (pu^\prime) u (\eta \bar{\eta})^\prime =  \int_a^b (pu^\prime) u\, (|\eta|^2)^\prime \nonumber \\ 
= - \int_a^b p(u^\prime)^2 |\eta|^2 - \int_a^b u|\eta|^2 (pu^\prime)^\prime. \label{e10}
\end{gather}
Integrating \eqref{e9}, using \eqref{e10} and \eqref{e1} once again, we obtain,
\begin{gather}\label{e11}
\int_a^b p|\varphi^\prime|^2 +q|\varphi|^2  = \int_a^b p u^2 |\eta^\prime|^2 + \lambda_u \int_a^b |\varphi|^2 w .
\end{gather}
from which \eqref{e8} follows immediately. 

Observe that the integral $\int_a^b p u^2 |\eta^\prime|^2 = \int_a^b p u^2 |\left(\frac{\varphi}{u}\right)^\prime|^2$ converges and is zero iff $u^2 |\left(\frac{\varphi}{u}\right)^\prime|^2 =0$ or, equivalently, if $|u \varphi^\prime - u^\prime \varphi|^2 =0$ and the conclusion follows.
\end{proof}
\begin{corollary}\label{cor1}
Under the assumptions of the previous lemma, 
\begin{equation}\label{e11}
\min_{\varphi \in \mathcal{V}} \left (\int_a^b p|\varphi^\prime|^2 +q|\varphi|^2 - \lambda_u \int_a^b |\varphi|^2\, w \right ) \geq 0
\end{equation}
where equality holds iff $\varphi =\alpha u$ for some constant $\alpha \in \C$.
\end{corollary}
\begin{remark}
Note that there is no restriction on the sign of the integral on the right of \eqref{e8} or \eqref{e11}. Indeed, the quantities may even be ghosts (real or complex).
\end{remark}

\section{The main results}\label{main}
In \cite{alm1}, \cite{alm2} it was shown that the mere existence of a non-real eigenvalue of \eqref{e1} must imply the non-existence of a real ground state (that is, an eigenfunction of \eqref{e1} having no zeros in $(a, b)$). This result was subsequently generalized in \cite{abm3} so as to simultaneously include analogous results for difference equations and partial differential equations of elliptic type. In addition, we have extended the result in \cite{alm2} below in Theorems~\ref{th1}-\ref{th3}. 

\begin{theorem}\label{th1}
Let $n \geq 1$ and let $\{\lambda_m, \bar{\lambda}_m\}_{m=1}^n$ be a set of $n$ mutually distinct pairs ($\lambda_i \neq  \lambda_j$ for $i \neq j$) of non-conjugate ($\lambda_i\neq \bar{\lambda}_j$ for $1 \leq i , j \leq n$) non-real eigenvalues of \eqref{e1}. Then  \eqref{e1} has no real eigenfunction with $n-1$ zeros in $(a, b)$.
\end{theorem}

\begin{proof}
Assume, on the contrary, that there exists an eigenfunction, $u$, corresponding to a real eigenvalue $\lambda_u$ having $n-1$ zeros $x_1< x_2< \ldots < x_{n-1}$ in $(a, b)$. Fix a set of eigenfunctions, $\varphi_i$, $i=1,2,\ldots, n,$ corresponding to the eigenvalues $\lambda_i$ and  choose the constants $c_i\in \C$ such that 
\begin{equation}\label{e12}
\varphi (x) = c_1 \varphi_1(x) + c_2\varphi_2(x)+ \cdots +\varphi_n(x)
\end{equation}
 satisfies $\varphi(x_i) =0$ for $i=1,2,\ldots, n-1$. Since there are $n-1$ equations in $n$ unknowns we can find and then fix such a set $c_i$ of constants, not all zero, and generally in $\C$. Using the notation of  Lemma~\ref{lem4} we observe that since $\varphi$ and $u$ have common zeros whose derivatives (here, $p\varphi^\prime, pu^\prime$) cannot vanish at these zeros, $\varphi \in \mathcal{V}$ where $\mathcal{V}$ is defined above. 

A long, straightforward calculation, and repeated uses of Lemma~\ref{lem1} and Lemma~\ref{lem2} yields \eqref{e5} for our choice of $\varphi$. An application of Lemma~\ref{lem4} gives a case of equality in \eqref{e8}  so that $\varphi(x) = \alpha u(x)$, where $\alpha \in \C$ is a non-zero constant. Unlike the case where $n=1$ in \cite{alm1} the contradiction here is not immediate as it is possible for $\phi(x)$, being a linear combination of non-real eigenfunctions,  to be real valued. 

We set aside this possibility by noting the following:  We know that $L\varphi(x) = \alpha Lu(x)$ and $L\varphi_i(x) = \lambda_i w(x)\varphi_i(x)$,  $Lu(x) = \lambda_u w(x) u(x)$ after which an application of the operator $L$ in \eqref{e1} to \eqref{e12}  gives us \\ 
$\sum_{m=1}^n c_m \lambda_m w \varphi_m  = \lambda_u w u$ or 
\begin{equation}\label{e13}
\sum_{m=1}^n c_m \lambda_m  \varphi_m  = \alpha \lambda_u  u
\end{equation}
on $\supp w$. On the other hand, since $\varphi(x)= \alpha u(x)$ (and $\varphi$ is given by \eqref{e12}) we get,
\begin{equation}\label{e14}
\sum_{m=1}^n c_m \varphi_m \lambda_u  = \alpha \lambda_u  u.
\end{equation}
Subtracting \eqref{e14} from \eqref{e13} we find,
\begin{equation}\label{e15}
\sum_{m=1}^n c_m (\lambda_m - \lambda_u ) \varphi_m  = 0,
\end{equation}
over $\supp w$. However,  $\lambda_m -\lambda_u \neq 0$ for all $m$, by hypothesis, and so Lemma~\ref{lem3} implies that all $c_m=0$, a contradiction.
\end{proof}

\begin{corollary}{\rm \label{cor2} [\cite{alm2}, Theorem 2.0]} 
If \eqref{e1} has a non-real eigenvalue there is no real eigenvalue whose eigenfunction is non-zero on $(a, b)$.
\end{corollary}

\begin{corollary}\label{cor2a} 
If \eqref{e1} has a non-real eigenvalue then $n_R = 1$, that is, the Richardson index is equal to 1.
\end{corollary}

The previous corollary covers the next numerical example.
\begin{example}[\cite{aj}, p. 39]
Let $p(x)=1$, $q(x)=q$, $w(x)=\sgn x$, in \eqref{e1}, for $x\in [-1,1]$. Then for $q =3$, \eqref{e1} has a pair of pure imaginary eigenvalues and there is no nonzero real eigenfunction on $(-1, 1)$ as the smaller oscillation number obtained there is one (1).
\end{example}

The next result is a reformulation of Theorem~\ref{th1} that also includes some real ghosts.

\begin{theorem}\label{th2}
Let $n \geq 1$ and let $\lambda_i \neq  \lambda_j$ for $i \neq j$ and $1 \leq i , j \leq n$ be real eigenvalues of \eqref{e1} whose eigenfunctions are all degenerate real ghosts. Then  \eqref{e1} has no real eigenfunction with $n-1$ zeros in $(a, b)$.
\end{theorem}

\begin{proof} The proof follows along the same lines as that of Theorem~\ref{th1}. We leave the details to the reader.
\end{proof}

\begin{example}[\cite{aj}, p. 41]
This is an application of Theorem~\ref{th2}. Let $p(x)=1$, $q(x)=q$, $w(x)=\sgn x$, in \eqref{e1}, for $x\in [-1,1]$. Then for $q \approx 21.99604$, \eqref{e1} has two real eigenvalues whose eigenvectors are (real) ghosts and there is no real eigenfunction having only one zero in $(-1,1)$.
\end{example}

The next example is covered by the case $n=2$ in Theorem~\ref{th1}.
\begin{example}[\cite{aj}, p. 40]
Let $p(x)=1$, $q(x)=q$, $w(x)=\sgn x$, in \eqref{e1}, for $x\in [-1,1]$. Then for $q = 15$, \eqref{e1} has a two pairs of non-real, non-conjugate  eigenvalues and there is no real eigenfunction having only one zero in $(-1,1)$ (so the smallest oscillation count is 2)
\end{example}
The final example is covered by the case $n=3$ in Theorem~\ref{th1}.
\begin{example}[\cite{aj}, p. 42]
Let $p(x)=1$, $q(x)=q$, $w(x)=\sgn x$, in \eqref{e1}, for $x\in [-1,1]$. Then for $q = 33$, \eqref{e1} has a three pairs of non-real, non-conjugate  eigenvalues and there is no real eigenfunction having only two zeros in $(-1,1)$ (so the smallest oscillation count is 3).
\end{example}

These examples indicate that it is very likely the case that one cannot do better than the estimates obtained in Theorem~\ref{th1}. Interestingly, both Theorem~\ref{th1} and Theorem~\ref{th2} can be combined so as to obtain the more general

\begin{theorem}\label{th3}
Let $m, n \geq 1$ and let $\{\lambda_k, \bar{\lambda}_k\}_{k=1}^m$ be a set of $m$ mutually distinct pairs ($\lambda_i \neq  \lambda_j$ for $i \neq j$) of non-conjugate ($\lambda_i\neq \bar{\lambda}_j$ for $1 \leq i , j \leq n$) non-real eigenvalues of \eqref{e1}. Let $\{\mu_i\}_{i=1}^n$ be a set of mutually distinct real eigenvalues corresponding to degenerate real ghosts. Then  \eqref{e1} has no real eigenfunction with $m+n-1$ zeros in $(a, b)$.
\end{theorem}
\begin{proof} The proof follows along the same lines as that of Theorem~\ref{th1} since, by definition,  degenerate real ghosts also satisfy the second of \eqref{e5}. We leave the details to the reader.
\end{proof}
Thus, by Theorem~\ref{th3}, the Richardson index satisfies,
\begin{equation}\label{ri}
 n_R \geq m+n.
 \end{equation}
That $n_R \neq 0$ necessarily was known to both Haupt \cite{oh} and Richardson \cite{rgdr} in the so-called {\it orthogonal} case (called {\it right definite} nowadays). In the special orthogonal case these results in \cite{oh}, \cite{rgdr} were rediscovered and improved upon in \cite{ekz}.


\begin{remark}
For $n=0, m=1$ we recover the main result in \cite{alm1}, \cite{alm2}. In addition, the estimate  $m+n$ in the previous theorem is best possible in the sense that non-degenerate real ghosts cannot be added to $m+n$ and maintain the conclusion.

This can be seen by the following example, based on an earlier one in [\cite{abm2}, Remark 4]. In that example $w(x)=1, x\in [0,1]$, $w(x)= -1$ on $[1,4]$, and $q(x) = - 9\pi^2/4, x\in [0,4]$. Extensive numerical calculations show that this problem has 4 non-degenerate real ghosts (at $\lambda \approx 1, 12.7, 18.8, 22.1)$ with each ghost carrying an oscillation number equal to $5, 4, 3, 2$ with $2$ being the smallest oscillation count. The next eigenvalue, i.e., $\lambda \approx 49.3$, has an eigenfunction with 3 zeros and the number of zeros increases indefinitely after that.  There are no degenerate real ghosts, and exactly 2 pairs of non-real eigenvalues (at $5.8 \pm 8.2 i$ and $-12\pm 4.1i$). If we were to add the non-degenerate ghosts to $m+n$ we would conclude that \eqref{e1} can have no real eigenfunction with $6$ or less zeros in $(a, b)$. However, this is false as this problem admits $\lambda=1$ as an eigenvalue with the non-degenerate real ghost $y(x) = \sin (3\pi x/2)$, an eigenfunction with five zeros in $(0, 4)$.
\end{remark}

\begin{example}
Let $p(x)=1$, $q(x)=q$, $w(x)=\sgn x$, in \eqref{e1}, for $x\in [-1,1]$. Then for $q = 4{\pi}^2$, \eqref{e1} has a 2 pairs of complex ghosts and one degenerate real ghost. Thus, $n_R \geq 3$ and there is no real eigenfunction having only two zeros in $(-1,1)$ as in the example.
\end{example}

Next, we give a minor but useful extension of the classical Liapunov inequality from the case of continuous to measurable coefficients. The proof is an adaptation of the method in \cite{wl} to the more general case considered here. This is followed by a more general result first due to Rapoport in \cite{imr} and cited by Kre\v{i}n \cite{mk} on its extension to solutions of periodic differential equations with a prescribed number of zeros. Our proof of the main result in \cite{imr} is simpler and uses induction on the number of zeros of a given solution. Using these estimates we can then associate the Richardson index to the Richardson number. In the following discussion all integrals are in the sense of Lebesgue. In the sequel all solutions are non-trivial unless otherwise specified. 

\begin{lemma}\label{liap} Let $q\in L^1(I)$ be real valued, $I=[a, b]$, and let $y$ be a solution of 
\begin{equation}\label{e16}
y^{\prime\prime}+q(x)y = 0
\end{equation}
having two consecutive zeros in $I$. Then
\begin{equation}\label{e17}
\int_a^b q_{+}(x)\ dx \geq \frac{4}{b-a},
\end{equation}
where $q_{+}(x) = (q(x)+|q(x)|)/2$ is the positive part of $q$.
\end{lemma}

\begin{proof} First, observe that the left hand side of \eqref{e17} cannot be zero since, if it were, this would imply that $q(x)\leq 0$ a.e. which, in turn, would imply that \eqref{e16} is disconjugate, \cite{ph}, i.e., no non-trivial solution can have more than one zero, contrary to the assumption. 

Since $q_{+}(x) \geq q(x)$ Sturm's comparison theorem implies that there is a solution $y$ of 
\begin{equation}\label{e18}
y^{\prime\prime}+q_{+}(x)y = 0
\end{equation}
such that $y(a)=0$, and $y(c)=0$, for some $c\in (a, b)$. Without loss of generality we may assume that $y(x)>0$ in $(a, c)$. Since $q$ is integrable over $I$ standard existence and uniqueness theorems under Carath\'{e}dory conditions \cite{er} imply that solutions to initial value problems associated with \eqref{e16} or \eqref{e18} are always absolutely continuous along with their first derivatives (and so in $C^1$). 

We show that $y$ cannot have a relative minimum in $(a, c)$. Assume, on the contrary, that $y$ does attain a relative minimum at $x=\gamma$, say $\gamma\in (a, c)$. Since $y$ is differentiable at $\gamma$, $y^{\prime}(\gamma)=0$. Thus, there is a $\delta >0$ such that $y^{\prime}(x)\leq 0$ for $x \in (\gamma-\delta, \gamma)$ and $y^{\prime}(x)\geq 0$ for $x \in (\gamma), \gamma+\delta)$. However $y^{\prime\prime}(x) \leq 0$ in $(a, c)$, thus, for $x \in (\gamma-\delta, \gamma)$ we must have $\int_{x}^{\gamma} y^{\prime\prime}(t)\, dt \leq 0$ i.e., $y^{\prime}(x)\geq 0$ there. Hence $y^{\prime}(x) =0$ a.e. in $(\gamma-\delta, \gamma)$. Similarly we can show that $y^{\prime}(x) =0$ in $(\gamma, \gamma +\delta)$.  Since $y^\prime$ is continuous in this neighborhood of $\gamma$, we must have $y(x)=C$, where $C$ is a constant, throughout $(\gamma-\delta, \gamma+\delta)$. Assuming there is a non-trivial solution $z$ such that $z(x)=C$ in $(\gamma-\delta, \gamma+\delta)$, then $z(\gamma)=C$ and $z^\prime(\gamma)=0$. Their difference $w=y-z$ is a solution of \eqref{e18} that satisfies $w(\gamma)=0, w^\prime(\gamma)=0$. Hence $w\equiv 0$, i.e., $y=z$ and so $y(x)=C$ throughout $(a, c)$ i.e, $y(x)\equiv 0$ throughout $[a, c]$, and therefore throughout $[a, b]$, which is impossible, again by uniqueness. Thus, our solution $y$ cannot have a relative minimum in $(a, c)$.

Now let $y$ attain its absolute maximum at $x=m$, say, $m\in (a, c)$. Since $y^\prime$ is not increasing there 
follows that $y^\prime(a) \geq y(m)/(m-a)$. Similarly, $y^\prime(c) \geq y(m)/(m-c)$. Thus,
$$\frac{-y^\prime(c)+y^\prime(a)}{y(m)} \geq \frac{1}{m-a}+\frac{1}{c-m} \geq  \frac{1}{m-a}+\frac{1}{b-m}=\frac{b-a}{(m-a)(b-m)}.$$
Since $y^\prime$ is absolutely continuous, we get (suppressing the variables of integration), 
$$ \frac{-y^\prime(c)+y^\prime(a)}{y(m)}= \frac{-1}{y(m)}\int_a^c y^{\prime\prime}\, dx = \frac{1}{y(m)}\int_a^c q_{+}\, y\, dx \geq \frac{b-a}{(m-a)(b-m)}.$$
However, $$ \frac{1}{y(m)}\int_a^c q_{+}\, y\, dx \leq \int_a^c q_{+}\, dx,$$
by construction. Therefore, 
$$\int_a^c q_{+}\, dx \geq \frac{b-a}{(m-a)(b-m)} \geq \inf_{x \in [a, b]} \frac{b-a}{(x-a)(b-x)} = \frac{4}{b-a}$$
and the result follows.
\end{proof}
We formulate a slightly more general result in the case our solution has many zeros (see \cite{imr} for the case where $q(x)\geq 0$ and has mean value equal to 1).

\begin{lemma}\label{rap} Let $q\in L^1(I)$ be real valued, $I=[a, b]$, and let $y$ be a solution of \eqref{e16} having $n$ interior zeros in $I$, $n\geq 0$, and vanishing at the endpoints. Then
\begin{equation}\label{e20}
\int_a^b q_{+}(x)\ dx \geq \frac{4(n+1)^2}{b-a},
\end{equation}
where $q_{+}(x) = (q(x)+|q(x)|)/2$ is the positive part of $q$.
\end{lemma}
\begin{proof} As in Lemma~\ref{liap} we assume that $y(a)=y(b)=0$. Let $x_0=a< x_1 < x_2 < x_3 < \cdots < x_n < b=x_{n+1}$ denote the $n$ interior zeros of $y$. The case $n=1$ is a consequence of Lemma~\ref{liap}. Since, for every $m=0, 1, 2, \ldots, n$, we have $y(x_m)=y(x_{m+1})=0$, Lemma~\ref{liap} implies that 
\begin{equation}\label{ih1}
\int_{x_m}^{x_{m+1}} q_{+}(x)\, dx \geq \frac{4}{x_{m+1}-x_m}.
\end{equation}
for each such $m$. 
Therefore,
\begin{equation}\label{ih2}
\int_a^{b} q_{+}(x)\, dx \geq \sum_{m=0}^n \frac{4}{x_{m+1}-x_m}.
 \end{equation}
Writing
$$ \sum_{m=0}^n \frac{1}{t_{m+1}-t_m} \equiv f(t_1, t_2, \ldots, t_n)$$
where $t_0=a, t_{n+1}=b$, we proceed to calculate the minimum of $f$ over the bounded set $V$ of points $(t_1, t_2, \ldots,t_n)$ satisfying
$$a< t_1 < t_2 < t_3 < \cdots < t_n < b.$$
A straightforward calculation gives that $\nabla f ={\bf 0}$ iff
\begin{equation}\label{ih3}
(t_1-a)^2=(t_2-t_1)^2 = \cdots = (t_n-t_{n-1})^2 = (b-t_n)^2.
\end{equation}
Taking into account the fact that the $t_m$ are all distinct and increasing we get a system of $n$ equations in $n$ unknowns (recall that $t_0=a, t_{n+1}=b$) i.e.,
$$t_m -2t_{m-1}+t_{m-2}=0, \quad m=0,1,2,\ldots, n+1.$$
The system may be solved recursively yielding the solutions
\begin{equation}\label{ih4}
t^*_m  = \frac{(n-m+1)a+mb}{n+1} , \quad m=1, 2, \ldots, n.
\end{equation}
Finally, a simple though lengthy calculation gives that
$$ \inf_{(t_1, t_2, \ldots,t_n)\in V} f(t_1, t_2, \ldots, t_n) = f(t^*_1, t^*_2, \ldots, t^*_n) = \frac{(n+1)^2}{b-a}.$$
Equation \eqref{ih2} now implies that
$$\int_a^{b} q_{+}(x)\, dx \geq 4 f(t^*_1, t^*_2, \ldots, t^*_n) = \frac{4(n+1)^2}{b-a},$$
as required.
\end{proof}
These two lemmas can be used to give bounds on both the Richardson, $n_R$, and Haupt, $n_H$, indices.

\begin{theorem}\label{th4} Let $p\equiv 1$, $w, q \in L^1(a, b)$ and $\int_a^b w_+(x)\, dx >0$.
For $\lambda > 0$, let there be an eigenfunction corresponding to a real eigenvalue $\lambda_{n_R}$ of \eqref{e1} having $n_R$ zeros in $(a, b)$ (where we recall, by definition, that there cannot be any real eigenfunction having fewer than $n_R$ zeros in $(a, b)$). Then, $$\lambda_{n_R} \leq  \Lambda_R,$$ $$\lambda_{n_R} =  \Lambda_H,$$ where $\Lambda_H$ is the Haupt number  and
$$n_R + 1 \leq \left (\frac{b-a}{4} \left \{ \Lambda_{H}\int_a^b w_+(x)\, dx + \int_a^b q_-(x) \,dx\right \} \right)^{1/2}.$$
\begin{proof}  Since the number of zeros of any real eigenfunction may increase and even decrease as the eigenvalue parameter increases  (and this only a finite number of times, see \cite{abm2}) we see that $\lambda_{n_R} \leq \Lambda_R$. That $\lambda_{n_R} =  \Lambda_H,$ is clear from the definitions. 

Using Lemma~\ref{rap} we find,
$$\int_a^b (\lambda_{n_R} w(x) - q(x) )_+ \, dx \geq \frac{4(n_R+1)^2}{b-a}.$$
However, for any $\lambda>0$ there holds, $(\lambda w(x) - q(x) )_+  \leq \lambda r_+(x)  +  q_-(x)$. tince $\lambda_{n_R} =  \Lambda_H,$ the result follows.
\end{proof}
\end{theorem}
Given the definition of the Haupt-Richardson indices and numbers we know that 
$n_R \leq n_H.$ The next result is an immediate consequence of Theorem~\ref{th4}.
\begin{theorem}\label{th5} Let $\lambda_{n_H}$ be the smallest positive eigenvalue whose eigenfunction has $n_H$ zeros in $(a, b)$. Then, $$\lambda_{n_H} = \Lambda_{R}$$
and 
$$n_H + 1 \leq \left (\frac{b-a}{4} \left \{ \Lambda_{R}\int_a^b w_+(x)\, dx + \int_a^b q_-(x) \,dx\right \} \right)^{1/2},$$
where $\Lambda_{R}$ is the Richardson number.
\end{theorem}

Theorems~\ref{th4}, \ref{th5} provide a lower bound on both $\Lambda_R$ and $\Lambda_H$. Other lower bounds were discussed in \cite{aj} but for particular cases of the weight function. 

Further lower bounds may be obtained indirectly using Sturm's comparison theorem. For example, for $\lambda > 0$, since $\lambda w(x) - q(x) \leq \lambda |w(x)|-q(x)$ we can compare \eqref{e1} with 
\begin{equation}\label{e1z}
 - (p(x)z^\prime)^\prime + q(x)\, z =\mu |w(x)|\, z,\quad \quad z(a) = 0 = z(b),
\end{equation}
for which theoretical lower bounds on its eigenvalues $\mu_n$ are many. Thus, if $\Lambda_R$ is the Richardson number for \eqref{e1}, then $\Lambda_R \geq \mu_{n_H}$ where $\mu_{n_H}$ is that eigenvalue of \eqref{e1z} whose eigenfunction, assuming it exists, has precisely $n_H$ zeros in $(a, b)$. 

{\bf Note:} The non-existence of such an eigenfunction is only possible in the case where $w(x)=0$ on one (or more) sets of positive Lebesgue measure since on such a set $E$, say, we may choose $q$  in such a way that all the solutions of \eqref{e1z} will have at least as many zeros on $[a, b]$ as the equation $ - (p(x)z^\prime)^\prime + q(x)\, z =0$ does for $x \in E$ (as $\mu$ will not appear in the equation but all eigenfunctions must have at least that many zeros). 

Setting aside such an anomaly for the moment and assuming $p(x)>0$, $w(x)>0$ and $p, q, w$ are all continuous on $[a, b]$ one can deduce that, for all $n \geq 0$, (see [\cite{bg}, Theorem 1]) ,
$$\mu_n > \inf_{x\in [a,b]} \frac{q(x)}{|w(x)|}+ \frac{(n+1)^2 \pi^2}{c^2 \left (\int_a^b  dx/p(x) \right)^2}$$
where $c^2=||wp||_{\infty}$, and the {\it inf} is assumed to exist. So, since $\Lambda_R \geq \mu_{n_H}$, we have
$$\Lambda_R > \inf_{x\in [a,b]} \frac{q(x)}{|w(x)|}+ \frac{(n_H+1)^2 \pi^2}{c^2 \left (\int_a^b  dx/p(x) \right)^2}.$$
The Haupt number, $\Lambda_{H}$, may be estimated similarly and we leave the details to the reader.

\begin{remark}Upper bounds on $\Lambda_R$ may be found in [\cite{aj}, Propositions 3, 4]. More results in this direction, along with numerical estimations, may be found in \cite{km}, \cite{mkm}, in the case where $w(x)$ has two turning points (i.e., sign changes) in $(a, b)$. 
\end{remark}

\begin{example}
In [\cite{aj}, p.41] there are numerical calculations indicating that in the case where $p(x)=1$, $q(x)= - 22$ for all $x\in [a, b]= [-1, 1]$, $w(x) = \sgn x$, the Richardson index, $n_R=2$, the Haupt index, $n_H=3$, the Richardson number,  $\Lambda_{R} \approx 5.7069$. The upper bound in Theorem~\ref{th4} gives $n_R \leq 5.845$, in agreement with the numerics.
\end{example}

\begin{example}
In the case where $p(x)=1$, $q(x) = - 41.9$ for all $x\in [a, b]= [-1, 1]$, $w(x) = \sgn x$ in [\cite{aj}, p.38], the Richardson index, $n_R=3$. Here, the smallest positive eigenvalue having an eigenfunction with 3 zeros is $\lambda_{n_R} \approx 23.3372$.  The upper bound in Theorem~\ref{th4} now gives $n_R \leq 8.771$, also in agreement with the numerics.
\end{example}

\section{Conclusions}

We have estimated the Haupt index, Richardson index, Haupt number, and Richardson number associated with non-definite Sturm-Liouville problems \cite{abm}.  In so doing, we have shown that the larger the number of non-real pairs of eigenvalues (and degenerate real ghosts) the larger the smallest oscillation number of the (real) eigenfunctions corresponding to real eigenvalues thereby proving a conjecture in \cite{atm2}. This {\it loss} of oscillation numbers provides a sharp contrast to Sturm's classical oscillation theorem. Thus, we have complemented the original estimates in [\cite{alm2}, Theorem 3.1]  with the estimates in the main theorems herein.

\section{Acknowledgement} 

The author should like to thank the referee for useful suggestions.

\end{document}